\documentclass[12pt, a4paper]{amsart}
\usepackage{amsmath,amssymb,amsthm,booktabs,longtable,enumerate}
\usepackage{comment}
\usepackage{graphicx}
\usepackage[all]{xy}
\usepackage{caption}
\theoremstyle{plain}
\makeatletter 

\@addtoreset{table}{section}
\makeatother %
\newtheorem{theorem}{Theorem}[section]
\newtheorem{lemma}[theorem]{Lemma}
\newtheorem{corollary}[theorem]{Corollary}
\newtheorem{proposition}[theorem]{Proposition}

\newtheorem{problem}[theorem]{Problem}
\newtheorem{remark}[theorem]{Remark}

\newcommand{\Ad}{\mathop{\mathrm{Ad}}\nolimits}

\newcommand{\Isom}{\mathop{\mathrm{Isom}}\nolimits}

\newcommand{\rank}{\mathop{\mathrm{rank}}\nolimits}

\newcommand{\diag}{\mathop{\mathrm{diag}}\nolimits}

\newcommand{\R}{\mathbb{R}}
\newcommand{\C}{\mathbb{C}}

\newcommand{\Z}{\mathbb{Z}}

\makeatletter
    
    \@addtoreset{equation}{section}
  \makeatother
\usepackage{enumerate}%
\usepackage[
bookmarks=true,bookmarksnumbered=true,
bookmarksopen=true]{hyperref}
\makeatother
\begin{document}

\title[Kobayashi's properness criterion]{A proof of  Kobayashi's properness criterion from a viewpoint of metric geometry}
\author{Kento Ogawa, Takayuki Okuda}
\subjclass[2020]{
Primary 57S30, 
Secondary 
22E40, 
22F30, 
53C35, 
53C23} 
\keywords{proper action; discontinuous group; symmetric space; CAT(0) space}
\thanks{The first author is supported by JST SPRING, Grant Number JP-MJSP2132. 
The second author is supported by JSPS Grants-in-Aid for Scientific Research JP20K03589, JP20K14310, and JP22H0112}

\address[K.~Ogawa]{%
	Graduate School of Advanced Science and Engineering, Hiroshima University, 
    1-3-1 Kagamiyama, Higashi-Hiroshima City, Hiroshima, 739-8526, Japan.
        }
\email{knt-ogawa@hiroshima-u.ac.jp}
\address[T.~Okuda]{%
	Graduate School of Advanced Science and Engineering, Hiroshima University, 
    1-3-1 Kagamiyama, Higashi-Hiroshima City, Hiroshima, 739-8526, Japan.
        }
\email{okudatak@hiroshima-u.ac.jp}

\maketitle

\begin{abstract}
Let $G$ be a locally-compact group 
and $(H,L)$ a pair of closed subgroups of $G$.
For the cases where $G$ is a real linear reductive Lie group, 
T.~Kobayashi [Math.~Ann.~'89, ~ J.~Lie Theory '96] established a criterion for properness of 
the $L$-action on the homogeneous space $G/H$ 
in terms of Cartan's KAK-decomposition of $G$.
In this paper, 
we show that a similar theorem 
also holds 
if $G$ is a locally-compact group admitting a suitable isometric action on a metric space,
and give a proof of Kobayashi's criterion in terms of CAT(0) metric geometry on non-compact Riemannian symmetric spaces.
\end{abstract}

\section{Introduction}\label{section:introduction}

Let $G$ be a locally-compact (Hausdorff) group and $(H_1,H_2)$ a pair of closed subgroups of $G$.
The purpose of this paper is to give a 
 criterion for properness of the $H_2$-action on the homogeneous space $G/{H_1}$ in the cases where $G$ admits a suitable isometric action on a metric space.
In particular, we also give a proof 
 of T.~Kobayashi's properness criterion in \cite{Kobayashi89,Kobayashi96} from the viewpoint of CAT(0) metric geometry on non-compact Riemannian symmetric spaces.

First, we recall that a continuous action of a locally-compact group $H$
on a locally-compact Hausdorff space $X$ is said to be proper 
if for any compact subset $D$ of $X$, 
the closed subset 
\[
H_D := \{ h \in H \mid (h \cdot D) \cap D \neq \emptyset \}
\]
of $H$ is compact.

We motivated our work in the following research theme:
\begin{itemize}
    \item For a given locally-compact group $G$, study pairs $(H_1,H_2)$ of closed subgroups of $G$ such that the $H_2$-action on the homogeneous space $G/{H_1}$ is proper (or equivalently, the $H_1$-action on the homogeneous space $G/{H_2}$ is proper).
\end{itemize}

It should be noted that for a pair of closed subgroups $(H_1,H_2)$ of $G$, if one of $H_1$ and $H_2$ is compact, 
then the $H_2$-action on $G/{H_1}$ should be proper.
However, in the cases where $H_1$ and $H_2$ are both non-compact, 
the action is not needed to be proper. 
For an extremal example, 
if we take $(G,H_1) = (SO_0(n + 1,1), SO_0(n,1))$,
then for any non-compact closed subgroup $H_2$ of $G$,
the $H_2$-action on $G/{H_1}$ fails to be proper (Calabi--Markus phenomenon \cite{CalabiMarkus1962}). 

A systematic investigation of the theme above for non-compact closed subgroups $(H_1,H_2)$ in a reductive group $G$ was initiated by T.~Kobayashi's works \cite{Kobayashi89,Kobayashi92Fuji,Kobayashi1992necessary} in the 1980s, 
as studies of discontinuous groups and Clifford--Klein forms for homogeneous spaces $G/H$ of reductive type, and has been developed by many researchers in various approaches (see summaries \cite{Kobayashi-unlimit} and \cite[Section 4]{Kobayashi22conjecture} given by T.~Kobayashi for the details of discontinuous groups and Clifford--Klein forms).

In this paper, 
as one of the fundamental problems in the research theme above, 
we focus on the following problem:

\begin{problem}
For a given locally-compact group $G$ and a pair $(H_1,H_2)$ of closed subgroups of $G$, 
find a good criterion for properness of the $H_2$-action on the homogeneous space $G/{H_1}$.
\end{problem}

It should be emphasized that to prove the properness of the $H_2$-action on $G/{H_1}$, as we will see in Section \ref{subsection:properinG}, we need to check that the subset 
\[
H_1 \cap (D \cdot H_2 \cdot D^{-1})
\]
of $G$ is compact for any compact subset $D$ of $G$.
Such the problem might not be easy in many situations.

Now, for real linear reductive Lie group $G$, we state Kobayashi's properness criterion as follows:

\begin{theorem}[Kobayashi {\cite{Kobayashi89,Kobayashi96}}; see also Theorem \ref{theorem:Kobayashi96} for the detailed version]
\label{theorem:Kobayashi96intro}
Let $G$ be a real linear reductive Lie group 
and we fix a Cartan's KAK-decomposition $G = KAK$.
The Lie algebra of $A$ is denoted by $\mathfrak{a}$.
Then for a pair $(H_1,H_2)$ of closed subgroups of $G$, 
the $H_2$-action on $G/{H_1}$ is proper if and only if 
the intersection 
\[
\mathfrak{a}(H_1) \cap (\mathfrak{a}(H_2) + D)
\]
is compact for any compact subset $D$ of $\mathfrak{a}$,
where we put 
\[
\mathfrak{a}(H) := \{ X \in \mathfrak{a} \mid (K \cdot (\exp X) \cdot K) \cap H \neq \emptyset \} \subset \mathfrak{a}
\]
for each non-empty subset $H$ of $G$.
\end{theorem}

In the setting of Theorem \ref{theorem:Kobayashi96intro}, 
$\mathfrak{a}$ is just a finite-dimensional vector space, 
and $\mathfrak{a}(H_1)$ and $\mathfrak{a}(H_2)$ are both subsets of $\mathfrak{a}$.
Further, we only need to consider closed balls as compact subsets $D$ in $\mathfrak{a}$ (with respect to a fixed inner-product on $\mathfrak{a}$).
Therefore, the condition stated in Theorem \ref{theorem:Kobayashi96intro} is much easier than the original one.

In the cases where $H_1$ and $H_2$ are both reductive subgroups of $G$, 
then by taking inner-conjugations of $H_1$ and $H_2$ in $G$, 
one can choose $\mathfrak{a}(H_1)$ and $\mathfrak{a}(H_2)$ to be a finite union of linear subspaces of $\mathfrak{a}$ (see \cite{Kobayashi89} for more details),
and thus we do not need to consider compact subsets $D$ of $\mathfrak{a}$.
In particular, in such a situation, 
the $H_2$-action on $G/H_1$ is proper 
if and only if the $H_2$-action on $G/{H_1}$ has the property (CI), that is the isotropy subgroup of $H_2$ is compact at any point in $G/{H_1}$.

Kobayashi's properness criterion is a fundamental tool to study discontinuous groups and Clifford--Klein forms of homogeneous spaces $G/H$ of reductive type.
As one of important applications, 
a criterion for the Calabi--Markus phenomenon on $G/H$ was given in \cite[Corollary 4.3]{Kobayashi89},
that is, he showed that 
a homogeneous space $G/H$ of reductive type
admits an infinite discontinuous group if and only if  $\rank_\R G > \rank_\R H$.
Further, one can find plenty of applications of Kobayashi's properness criterion by many researchers (for example, \cite{BochenskiTralle2015ahyp, Kannaka2021linear, Kannaka2021quantitative, Kassel2008corankone, Kassel12, Kobayashi1998deformation,KobayashiYoshino05, 
Morita2022Cartan, OhWette2000new, Okuda13,Salein2000,Tojo2019classifcation}).

It should be remarked that Y.~Benoist \cite{Benoist96} also proved a claim which is essentially the same as Theorem \ref{theorem:Kobayashi96intro} 
inspired by Kobayashi's earlier work in \cite{Kobayashi89}.
Furthermore, he also gave a similar criterion for the cases where $G$ is a semisimple algebraic group over a non-Archimedean local fields (See also Remark \ref{remark:BenoistBuilding}).

We also give a remark as follows:
Let $K$ be a closed Lie subgroup of 
the orthogonal group $O(n)$,
and define the semi-direct product group $G := K \ltimes \R^n$.
Then some types of properness criterions 
are given by T.~Yoshino \cite{Yoshino2007CartanMotion} for closed subgroups of $G$ in the cases where $G$ is a Cartan motion group associated to a real reductive Lie group, 
and by A.~Baklouti, S.~Bejar and R.~Fendri \cite{BakloutiBejarFendri2023CptExtRn} for general $G = K \ltimes \R^n$ and pairs of connected or discrete subgroups of $G$ (see also Remark \ref{remark:YBBF} in Section \ref{subsection:main}).

As other situations, studies of relationships between properness and the property (CI) for pairs of closed subgroups of solvable or nilpotent Lie groups can be found in \cite{BakloutiKhlif2005ExpSolv, BakloutiKhlif2007weak, Lipsman1995proper, Nasrin2001twostep,Yoshino2004fourdim, Yoshino2005counterLips, Yoshino2007threestep}.

We note that the proofs of Theorem \ref{theorem:Kobayashi96} given by Kobayashi in \cite{Kobayashi89, Kobayashi96}
are based on structure theory of real linear reductive Lie groups.
The purpose of this paper is to understand 
Theorem \ref{theorem:Kobayashi96} in terms of metric geometry.

In this paper, we say that a metric space $(M,d)$ is proper if any closed ball in $M$ is compact.
As our main results, we will prove claims similar to 
Theorem \ref{theorem:Kobayashi96intro} (see Theorems \ref{theorem:main} and \ref{theorem:main_section} in Section \ref{subsection:main}), 
in the cases where $G$ is a locally-compact group admitting a proper isometric transitive action on a proper metric space $(M,d)$ with the following property:
\begin{description}
    \item[Property (S)] 
    For each $r \geq 0$, there exists $\tau(r) \geq 0$ such that for any $p_0,p,q \in M$ with $d(p,p_0) \leq r$, one can choose an isometry $g \in G$ with $g \cdot p = p_0$ and $d(q,g \cdot q) \leq \tau(r)$.
\end{description}

Roughly, Property (S) on the $G$-space $M$ stated above means that 
for any triple $p_0,p,q$ of points in $M$, 
one can choose a point $q_0$ in $M$ 
such that the segments $pq$ and $p_0q_0$ are $G$-conjugate to each other 
and the distance $d(q,q_0)$ is bounded in some sense. 

We will give a sufficient condition on a metric $G$-space $M$ for Property (S) in terms of geometry of CAT(0) spaces (see Section \ref{section:sufficientP}).
Furthermore, Theorem \ref{theorem:Kobayashi96intro}
can be explained in terms of CAT(0) metric geometry on Riemannian symmetric spaces without compact factors (see Section \ref{section:KobayashiCriterion}).

\subsection{Organization of the paper}

This paper is organized as follows:
In Section \ref{section:KobayashiMain}, 
we set up our notations of proper pairs and HBI-pairs.
Further, we state Kobayashi's properness criterion in the form of \cite{Kobayashi96} and our main results.
In Section \ref{section:nonexpanding}, 
we consider non-expanding maps between proper metric spaces and study the relationship between HBI-pairs via the map.
In the same section, we also discuss that for an isometric action on a proper metric space of a compact group, 
the space of orbits can be considered as a proper metric space 
and the quotient map is non-expanding.
Our main theorems \ref{theorem:main} and \ref{theorem:main_section} will be proved in Section \ref{section:proof}.
In Section \ref{section:sufficientP}, 
a sufficient condition on a metric homogeneous space 
for Property (S) is given in terms of CAT(0) spaces.
Finally, in Section \ref{section:KobayashiCriterion}, 
we discuss that Kobayashi's properness criterion (Theorem \ref{theorem:Kobayashi96}) can be explained by combining Theorem \ref{theorem:main_section} with some facts for metric geometry on Riemannian symmetric spaces.

\section{Kobayashi's properness criterion and main results of this paper}\label{section:KobayashiMain}

In this section, we set up our terminologies and state Kobayashi's properness criterion in the form of  \cite{Kobayashi96} and the main results of this paper.

\subsection{Proper pairs in locally-compact groups}\label{subsection:properinG}

To state Kobayashi's properness criterion given in \cite{Kobayashi96}, 
we shall fix our terminologies as below:
Let $G$ be a locally-compact group.
For a subset $H$ of $G$ and a compact subset $D$ of $G$,
we put 
\begin{align*}
\overline{N}_D(H) &:= D \cdot H \cdot D^{-1} \\
    &= \{ d_1 \cdot h \cdot d_2^{-1} \in G \mid d_1,d_2 \in D, h \in H \}.
\end{align*}
For two non-empty subsets $H_1$ and $H_2$ of $G$, 
we denote by $H_1 \sim H_2$ if there exists a compact subset $D$ of $G$ such that 
\[
H_1 \subset \overline{N}_D(H_2) \text{ and }
H_2 \subset \overline{N}_D(H_1). 
\]
Then ``$\sim$'' defines an equivalent binary relation 
on the set $\mathcal{P}^\times(G)$ of all non-empty subsets of $G$.
The equivalent class of $H \in \mathcal{P}^\times(G)$ will be denoted by $[H]$,
and the quotient set of $\mathcal{P}^\times(G)$ by the equivalent relation $\sim$ is written as $[\mathcal{P}^\times(G)]$.

The following proposition is pointed out by Kobayashi \cite{Kobayashi89,Kobayashi92Fuji}:

\begin{proposition}
Let $(H_1,H_2)$ be a pair of closed subgroups of $G$.
Then the following conditions on $(H_1,H_2)$ are equivalent: 
\begin{enumerate}
    \item The $H_2$-action on the homogeneous space $G/{H_1}$ is proper. 
    \item The $H_1$-action on the homogeneous space $G/{H_2}$ is proper.
    \item The $(H_1 \times H_2)$-action on the group manifold $G = (G \times G)/{(\diag G)}$ is proper, 
    where $\diag G := \{ (g,g) \mid g \in G \}$.
    \item For any $H_1' \in [H_1]$ and any $H_2' \in [H_2]$, the intersection $H_1' \cap H_2'$ is relatively compact in $G$.
\end{enumerate}    
\end{proposition}

Motivated by the fact above, 
for an arbitrary pair $(H_1,H_2)$ of non-empty subsets of $G$, 
we say that $(H_1,H_2)$ is proper in $G$, and write 
\[
H_1 \pitchfork H_2 \text{ in } G \quad  (\text{ or } [H_1] \pitchfork [H_2] \text{ in } G)
\]
if for any $H_1' \in [H_1]$ and any $H_2' \in [H_2]$, the intersection $H_1' \cap H_2'$ is relatively-compact in $G$.

Note that the following 
proposition gives necessary and sufficient conditions for a pair $(H_1,H_2)$ to be proper in $G$:

\begin{proposition}\label{proposition:LLHBI}
    Let $H_1$ and $H_2$ be non-empty subsets of $G$.
    Then the following four conditions are equivalent:
    \begin{enumerate}
        \item\label{item:LLHBI:LLHBI} The pair $(H_1,H_2)$ is proper in $G$.
        \item\label{item:LLHBI:DLDL} For any compact subset $D$ of $G$, the intersection $\overline{N}_D(H_1) \cap H_2$ is relatively compact in $G$.
        \item\label{item:LLHBI:LDLD} For any compact subset $D$ of $G$, the intersection $H_1 \cap \overline{N}_D(H_2)$ is relatively compact in $G$.
        \item\label{item:LLHBI:DLDDLD} For any compact subsets $D_1$ and $D_2$ of $G$, the intersection $\overline{N}_{D_1}(H_1) \cap \overline{N}_{D_2}(H_2)$ is relatively compact in $G$.
    \end{enumerate}
\end{proposition}

\begin{proof}[Proof of Proposition \ref{proposition:LLHBI}]
The arguments in the proof of Proposition \ref{proposition:CI} also work to prove Proposition \ref{proposition:LLHBI} as below.
The equivalence \eqref{item:LLHBI:LLHBI} $\Leftrightarrow$ \eqref{item:LLHBI:DLDDLD}, 
the implications 
\eqref{item:LLHBI:LLHBI} $\Rightarrow$ \eqref{item:LLHBI:DLDL}
and 
\eqref{item:LLHBI:LLHBI} $\Rightarrow$ \eqref{item:LLHBI:LDLD}
are all easy to prove.
The implications
\eqref{item:LLHBI:DLDL} $\Rightarrow$ \eqref{item:LLHBI:DLDDLD}
and 
\eqref{item:LLHBI:LDLD} $\Rightarrow$ \eqref{item:LLHBI:DLDDLD}
comes immediately from 
the inclusions
\begin{align*}
\overline{N}_{D_1}(H_1) \cap \overline{N}_{D_2}(H_2) &\subset \overline{N}_{D_1}(H_1 \cap \overline{N}_{D_1^{-1} \cdot D_2}(H_2)), \\
\overline{N}_{D_1}(H_1) \cap \overline{N}_{D_2}(H_2) &\subset \overline{N}_{D_2}(\overline{N}_{D_2^{-1} \cdot D_1}(H_1) \cap H_2)
\end{align*}
for compact subsets $D_1$ and $D_2$ of $G$, respectively.
\end{proof}

For each non-empty subset $H$ of $G$, 
let us define 
\begin{align*}
\pitchfork(H:G) &:= \{ H' \in \mathcal{P}^\times(G) \mid 
H \pitchfork H' \text{ in } G \}, \\
\pitchfork([H]:G)
&:= \{ \mathcal{H}' \in [\mathcal{P}^\times(G)] \mid  [H] \pitchfork \mathcal{H}' \text{ in } G \}. 
\end{align*}

In the cases where $G$ is $\sigma$-compact, 
that is, $G$ is a union of a countable family of compact subsets of $G$, 
the following holds:

\begin{proposition}[Yoshino \cite{Yoshino2007duality}]\label{proposition:LHHBIdual}
    Suppose that $G$ is $\sigma$-compact.
    Then for two non-empty subsets $H_1$ and $H_2$ of $G$, 
    the three conditions below are equivalent:
    \begin{enumerate}
        \item \label{item:LHHBIdual:sim} $H_1 \sim H_2$ in $G$.
        \item \label{item:LHHBIdual:dual} $\pitchfork(H_1:G) = \pitchfork(H_2:G)$.
        \item \label{item:LHHBIdual:dualsim} $\pitchfork([H_1]:G) = \pitchfork([H_2]:G)$.
    \end{enumerate}
\end{proposition}

\begin{proof}[Proof of Proposition \ref{proposition:LHHBIdual}]
    The implication \eqref{item:LHHBIdual:sim} $\Rightarrow$ \eqref{item:LHHBIdual:dual} 
    and the equivalence \eqref{item:LHHBIdual:dual} $\Leftrightarrow$ \eqref{item:LHHBIdual:dualsim} are  both trivial.
    We shall prove the contraposition of the implication \eqref{item:LHHBIdual:dual} $\Rightarrow$ \eqref{item:LHHBIdual:sim}.
    Assume that $H_1 \not \sim H_2$.
    Without loss of the generalities, one can suppose that 
    \[
    H_1 \not \subset \overline{N}_{D}(H_2)
    \]
    for any compact subset $D$ of $G$.
    Since $G$ is $\sigma$-compact and locally-compact, as is well-known that one can find a sequence $\{ D_n \}_{n \in \Z_{\geq 0}}$ of compact subsets of $G$ satisfying that 
    \begin{itemize}
        \item $D_n$ is a neighborhood of the unit of $G$ for each $n$, 
        \item $D_n \subset D_{n+1}^{\circ}$ for each $n$, where $D_{n+1}^{\circ}$ denotes the interior of $D_{n+1}$ in $G$, and 
        \item $G = \bigcup_n D_n$.
    \end{itemize}
    Note that for any compact subset $D$ of $G$, there exists $k \in \Z_{\geq 0}$ with $D \subset D_k$ and $D \subset \overline{N}_{D_k}(H_2)$.
    Let us fix a point 
    \[
    g_n \in H_1 \setminus \overline{N}_{D_n}(H_2)
    \]
    for each $n \in \Z_{\geq 0}$, and define $H' := \{ g_n \mid n \in \Z_{\geq 0} \}$.
    Then the intersection 
    \[
    H' \cap H_1 = H' 
    \]
    is not relatively compact in $G$ since $H' \not \subset \overline{N}_{D_k}(H_2)$ for any $k \in \Z_{\geq 0}$.
    Therefore, the pair $(H_1,H')$ is not proper in $G$.
    On the other hand, for any compact subset $D$ of $G$, 
    the intersection $H' \cap \overline{N}_{D}(H_2)$ is compact since 
    \[
    H' \cap \overline{N}_{D}(H_2) \subset H' \cap \overline{N}_{D_k}(H_2) = \{ g_k \mid k \leq n \}.
    \]
    for some $k \in \Z_{\geq 0}$ with $D \subset D_k$.
    Thus, the pair $(H_2,H')$ is proper in $G$ by Proposition \ref{proposition:LLHBI}.
    This proves $\pitchfork(H_1:G) \neq \pitchfork(H_2:G)$.
\end{proof}

We also remark that 
the following proposition
gives a sufficient condition for $\sigma$-compactness of $G$:

\begin{proposition}\label{proposition:metric_sigma_compact}
Let $G$ be a locally-compact group.
If $G$ admits a continuous proper isometric action on a non-empty proper metric space $(M,d)$, then $G$ is $\sigma$-compact.
\end{proposition}

\begin{proof}[Proof of Proposition \ref{proposition:metric_sigma_compact}]
Fix a point $*$ of $M$ and define 
\[
\pi : G \rightarrow M, ~ g \mapsto g \cdot *.
\]
Then the map $\pi$ is proper.
Since $d$ is a proper metric on $M$, 
for each $r \geq 0$, 
the closed ball $\overline{N}_r(*)$ is compact,
and thus $D_r := \pi^{-1}(\overline{N}_r(*))$ too. 
This proves that 
\[
G = \bigcup_{n \in \Z_{\geq 0}} D_n
\]
is $\sigma$-compact.
\end{proof}

\subsection{Kobayashi's properness criterion}

We shall state Kobayashi's properness criterion given in \cite{Kobayashi96}.
Let $G$ be a real linear reductive Lie group and 
$\theta$ a Cartan involution of $G$.
We write $K := G^\theta$ for the fixed point subgroup of $G$ by $\theta$, which is a maximal compact subgroup of $G$.
For the Lie algebra $\mathfrak{g}$ of $G$, 
the Cartan decomposition on $\mathfrak{g}$ induced by $\theta$ is denoted by $\mathfrak{g} = \mathfrak{k} + \mathfrak{p}$.
We choose a maximal abelian subspace $\mathfrak{a}$ of $\mathfrak{p}$.
Then we have a Cartan's KAK-decomposition $G = K \cdot A \cdot K$, where $A := \exp(\mathfrak{a})$ denotes the connected abelian analytic subgroup of $G$ with respect to $\mathfrak{a}$.
We denote by $W := N_K(\mathfrak{a})/{Z_K(\mathfrak{a})}$ 
the Weyl group acting on $\mathfrak{a}$.
For each non-empty subset $H$ of $G$, 
we define the non-empty $W$-invariant subset $\mathfrak{a}(H)$ of $\mathfrak{a}$ by 
\[
\mathfrak{a}(H) := \{ X \in \mathfrak{a} \mid (K \cdot \exp(X) \cdot K) \cap H \neq \emptyset \}.
\]
Then $\mathfrak{a}(H)$ is a non-empty $W$-stable subset of $\mathfrak{a}$, 
and $H \sim \exp(\mathfrak{a}(H))$ in $G$.

Let us consider $\mathfrak{a}$ as a locally-compact additive group (with respect to the standard Hausdorff topology on $\mathfrak{a}$).
On the set $\mathcal{P}^\times(\mathfrak{a})^{W}$ of all non-empty $W$-invariant subsets of $\mathfrak{a}$,
we define the equivalent relation $\sim$ 
and proper pairs in $\mathcal{P}^\times(\mathfrak{a})^{W}$ as in the sense mentioned above.
For each $C \in \mathcal{P}^\times(\mathfrak{a})^W$, we define 
\begin{align*}
    \pitchfork(C:\mathfrak{a})^{W} := \{ C' \in \mathcal{P}^\times(\mathfrak{a})^W \mid C \pitchfork C' \text{ in } \mathfrak{a} \} \subset \mathcal{P}^\times(\mathfrak{a})^W.
\end{align*}

Then Kobayashi's properness criterion in \cite{Kobayashi96} can be stated as follows:

\begin{theorem}[Kobayashi {\cite[Theorems 3.4 and 5.6]{Kobayashi96}}]\label{theorem:Kobayashi96}
Let $(H_1,H_2)$ be a pair of non-empty subsets of a real linear reductive Lie group $G$.
\begin{enumerate}[(1)]
\item \label{item:Kobayashi96:1} The pair $(H_1,H_2)$ is proper in $G$
    if and only if the pair $(\mathfrak{a}(H_1),\mathfrak{a}(H_2))$ is proper in $\mathfrak{a}$, that is, 
    for any compact subset $D$ of $\mathfrak{a}$, 
    the intersection 
    \[
    \mathfrak{a}(H_1) \cap (\mathfrak{a}(H_2) + D)
    \]
    is relatively-compact in $\mathfrak{a}$.
    \item The following four conditions on $(H_1,H_2)$ are equivalent:
\begin{enumerate}[(\textrm{2}a)]
    \item $H_1 \sim H_2$ in $G$.
    \item $\mathfrak{a}(H_1) \sim \mathfrak{a}(H_2)$ in $\mathfrak{a}$.
    \item $\pitchfork(H_1:G) =  \pitchfork(H_1:G)$
    \item $\pitchfork(\mathfrak{a}(H_1):\mathfrak{a})^{W} =  \pitchfork(\mathfrak{a}(H_2):\mathfrak{a})^{W}$.
\end{enumerate}
\end{enumerate}
\end{theorem}

\subsection{HBI-pairs in proper metric spaces}

Let $(M,d)$ be a proper metric space,
that is, $(M,d)$ is a metric space such that 
any closed ball in $M$ is compact.
Note that a subset of $M$ is bounded in $M$ if and only if relatively-compact in $M$.

In this subsection, to state our main results, 
we introduce the concepts of ``HBI-pairs'' of non-empty subsets of $M$ as below:

For each subset $C$ of $M$ and each $r \geq 0$, we define the subset $\overline{N}_r(C)$ of $M$ by
\[
\overline{N}_r(C) := \{ p \in M \mid d(p,c) \leq r \text{ for some } c \in C \}.
\]
In this paper, for two non-empty subsets $C_1$ and $C_2$ of $M$, we write $C_1 \sim C_2$ 
if there exists $r \geq 0$ such that 
\[
C_1 \subset \overline{N}_r(C_2) \text{ and } C_2 \subset \overline{N}_r(C_1).
\]
Note that $C_1 \sim C_2$ if and only if 
the Hausdorff distance between $C_1$ and $C_2$ is finite.
Then ``$\sim$'' defines 
an equivalent binary relation on the set $\mathcal{P}^\times(M)$ of all non-empty subsets of $M$.
For each non-empty subset $C$, we write $[C]$ for the equivalent class of $C$.
Furthermore, the quotient set of $\mathcal{P}^{\times}(M)$ by $\sim$ is denoted by $[\mathcal{P}^{\times}(M)]$.

Throughout this paper, 
for two non-empty subsets $C_1$ and $C_2$ of $M$,we say that the pair $(C_1,C_2)$ is \emph{HBI} in $M$ 
if for any $C_1' \in [C_1]$ and any $C_2' \in [C_2]$, 
the intersection $C_1' \cap C_2'$ is bounded in $M$.
In such a case, the pair $([C_1],[C_2])$ in $[\mathcal{P}^{\times}(M)]$ is also said to be HBI in $M$.

The proposition below gives necessary and sufficient conditions for a pair $(C_1,C_2)$ to be HBI in $M$.

\begin{proposition}\label{proposition:CI}
Let $C_1$ and $C_2$ be both non-empty subsets of $M$.
Then the following four conditions on $(C_1,C_2)$ are equivalent: 
\begin{enumerate}
    \item \label{item:CI:CI} The pair $(C_1,C_2)$ is HBI in $M$.
    \item \label{item:CI:C1NC2} For any $r \geq 0$, the intersection $C_1 \cap \overline{N}_r(C_2)$ is bounded in $M$.
    \item \label{item:CI:NC1C2} For any $r \geq 0$, the intersection $\overline{N}_r(C_1) \cap C_2$ is bounded in $M$.
    \item \label{item:CI:NC1NC2} For any $r_1,r_2 \geq 0$, the intersection $\overline{N}_{r_1}(C_1) \cap \overline{N}_{r_2}(C_2)$ is bounded in $M$.
\end{enumerate}    
\end{proposition}

\begin{proof}[Proof of Proposition \ref{proposition:CI}]
The equivalence \eqref{item:CI:CI} $\Leftrightarrow$ \eqref{item:CI:NC1NC2}, 
the implications 
\eqref{item:CI:CI} $\Rightarrow$ \eqref{item:CI:C1NC2}
and 
\eqref{item:CI:CI} $\Rightarrow$ \eqref{item:CI:NC1C2}
are all easy to prove.
The implications
\eqref{item:CI:C1NC2} $\Rightarrow$ \eqref{item:CI:NC1NC2}
and 
\eqref{item:CI:NC1C2} $\Rightarrow$ \eqref{item:CI:NC1NC2}
comes immediately from 
the inclusions
\begin{align*}
\overline{N}_{r_1}(C_1) \cap \overline{N}_{r_2}(C_2) &\subset \overline{N}_{r_1}(C_1 \cap \overline{N}_{r_1+r_2}(C_2)), \\
\overline{N}_{r_1}(C_1) \cap \overline{N}_{r_2}(C_2) &\subset \overline{N}_{r_2}(\overline{N}_{r_1+r_2}(C_1) \cap C_2)
\end{align*}
for $r_1,r_2 \geq 0$, respectively.
\end{proof}

For each non-empty subset $C$ of $M$, 
let us define 
\begin{align*}
\mathcal{HBI}(C:M) &:= \{ C' \in \mathcal{P}^\times(M) \mid 
\text{ the pair } (C,C') \text{ is HBI in } M \}, \\
\mathcal{HBI}([C]:M)
&:= \{ \mathcal{C}' \in [\mathcal{P}^\times(M)] \mid 
\text{ the pair } ([C],\mathcal{C}') \text{ is HBI in } M \}.
\end{align*}

Then, similarly to Proposition \ref{proposition:metric_sigma_compact}, 
the proposition below holds:

\begin{proposition}\label{proposition:CCdual}
    For two non-empty subsets $C_1$ and $C_2$ of $M$, 
    the following three conditions are equivalent:
    \begin{enumerate}
        \item \label{item:CCdual:sim} $C_1 \sim C_2$ in $M$.
        \item \label{item:CCdual:dual} $\mathcal{HBI}(C_1:M) = \mathcal{HBI}(C_2:M)$.             
        \item \label{item:CCdual:dualsim} $\mathcal{HBI}([C_1]:M) = \mathcal{HBI}([C_2]:M)$.
    \end{enumerate}
\end{proposition}

\begin{proof}[Proof of Proposition \ref{proposition:CCdual}]
The implication \eqref{item:CCdual:sim} $\Rightarrow$ \eqref{item:CCdual:dual}
and the equivalence \eqref{item:CCdual:dual} $\Leftrightarrow$ \eqref{item:CCdual:dualsim} are both trivial.
    We shall prove the contraposition of the implication \eqref{item:CCdual:dual} $\Rightarrow$
    \eqref{item:CCdual:sim}. 
    Assume that $C_1 \not \sim C_2$.
    Without loss of the generalities, one can suppose that 
    \[
    C_1 \not \subset \overline{N}_r(C_2)
    \]
    for any $r \geq 0$.
    For each $n \in \Z_{\geq 0}$, one can find and do take $p_n \in C_1 \setminus \overline{N}_n(C_2)$.
    Let us put $C' := \{ p_n \mid n \in \Z_{\geq 0} \}$.
    Then the intersection 
    \[
    C' \cap C_1 = C' 
    \]
    is not bounded in $M$.
    Therefore, the pair $(C_1,C')$ is not HBI in $M$.
    However, the intersection $C' \cap \overline{N}_r(C_2)$ is finite for any $r \geq 0$, and thus the pair $(C_2,C')$ is HBI in $M$ by Proposition \ref{proposition:CI}.
    This proves $\mathcal{HBI}(C_1:M) \neq \mathcal{HBI}(C_2:M)$.
\end{proof}

\subsection{Main results}\label{subsection:main}

We shall state our results as below:
Let $(M,d)$ be a proper metric space
 and $G$ a locally-compact group acting 
continuously, properly, isometrically and transitively on $M$.
As in Section \ref{section:introduction}, 
suppose that the metric $G$-space $M$ has the following property:
\begin{description}
    \item[Property (S)] 
    For each $r \geq 0$, there exists $\tau(r) \geq 0$ such that for any $p_0,p,q \in M$ with $d(p,p_0) \leq r$, one can choose an isometry $g \in G$ with $g \cdot p = p_0$ and $d(q,g \cdot q) \leq \tau(r)$.
\end{description}

Let us fix a base point $*$ of $M$ and denote by $K$ the compact isotropy subgroup of $G$ at the point $*$ of $M$.
The space of $K$-orbits in $M$ is written as $K \backslash M$ equipped with the proper metric $d^K$ induced by the metric $d$ on $M$ (see Section \ref{subsection:Korbits} for the definition of the metric $d^K$ on $K \backslash M$).
We define the map $\mu : G \rightarrow K \backslash M$ by 
\[
\mu : G \rightarrow K \backslash M, ~ g \mapsto K \cdot (g \cdot *).
\]

One of our main results is stated below:

\begin{theorem}\label{theorem:main}
In the setting above, 
let $H_1$ and $H_2$ be both non-empty subsets of $G$.
\begin{enumerate}[(1)]
    \item \label{item:main:proper} The pair $(H_1,H_2)$ is proper in $G$ if and only if 
    the pair $(\mu(H_1),\mu(H_2))$ is HBI in $K \backslash M$, 
    that is, for any $r \geq 0$, the intersection 
    \[
    \mu(H_1) \cap \overline{N}_r(\mu(H_2))
    \]
    is bounded in the proper metric space $K \backslash M$.    
    \item \label{item:main:sim} The following four conditions on $(H_1,H_2)$ are equivalent:
    \begin{enumerate}[(\textrm{2}-a)]
        \item $H_1 \sim H_2$ in $G$.
        \item $\mu(H_1) \sim \mu(H_2)$ in $K \backslash M$.
        \item $\pitchfork(H_1:G) 
        = \pitchfork(H_2:G)$. 
        \item $\mathcal{HBI}(\mu(H_1): K \backslash M) = \mathcal{HBI}(\mu(H_2):K \backslash M)$.
    \end{enumerate}
\end{enumerate}
\end{theorem}

In addition, let us fix $\Sigma$ as a closed subset of $M$ satisfying the following two conditions:
\begin{itemize}
    \item For each $K$-orbit $\mathcal{O}$ in $M$, the intersection $\Sigma \cap \mathcal{O}$ is non-empty.
    \item For each pair of $K$-orbits $(\mathcal{O}_1,\mathcal{O}_2)$ in $M$, 
    there exists $(p,q) \in (\Sigma \cap \mathcal{O}_1) \times (\Sigma \cap \mathcal{O}_2)$ such that $d(p,q) = d^K(\mathcal{O}_1,\mathcal{O}_2)$.
\end{itemize}
We write 
\[
\eta : \Sigma \rightarrow K \backslash M, ~ p \mapsto K \cdot p
\]
for the restriction of the quotient map, and define 
\[
\Theta : \mathcal{P}^\times(\Sigma) \rightarrow \mathcal{P}^\times(\Sigma), ~ S \mapsto \eta^{-1}(\eta(S)).
\]
Then $\Theta^2 = \Theta$ on $\mathcal{P}^\times(\Sigma)$.
Let us put 
\[
\mathcal{P}^\times(\Sigma)^{\Theta} := \{ S \in \mathcal{P}^{\times}(\Sigma) \mid \Theta(S) = S \} \subset \mathcal{P}^{\times}(\Sigma).
\]

By considering the restriction of $d$, the set $\Sigma$ itself is a proper metric space.
For each $S \in \mathcal{P}^\times(\Sigma)^{\Theta}$, 
we write
\[
\mathcal{HBI}(S:\Sigma)^{\Theta} := \{ S' \in \mathcal{P}^\times(\Sigma)^{\Theta} \mid \text{ the pair }(S,S') \text{ is HBI in } \Sigma \}.
\]

Furthermore, for each non-empty subset $H$ of $G$,
we define the non-empty $\Theta$-fixed subset $\Sigma(H)$ of $\Sigma$ by 
\begin{align*}
\Sigma(H) &:= \{ p \in \Sigma \mid (K \cdot p) \cap (H \cdot *) \neq \emptyset \}, \\
    &= \Sigma \cap (K \cdot H \cdot *).
\end{align*}

Then the following theorem is also holds:

\begin{theorem}\label{theorem:main_section}
In the setting above, 
let $H_1$ and $H_2$ be both non-empty subsets of $G$.
\begin{enumerate}[(1)]
\item \label{item:mainsec1} The pair $(H_1,H_2)$ is proper in $G$
    if and only if the pair $(\Sigma(H_1),\Sigma(H_2))$ is HBI in $\Sigma$.
    \item \label{item:mainsec2} The following four conditions on $(H_1,H_2)$ are equivalent:
\begin{enumerate}[(\textrm{2}-a)]
    \item $H_1 \sim H_2$ in $G$.
    \item $\Sigma(H_1) \sim \Sigma(H_2)$ in $\Sigma$.
    \item $\pitchfork(H_1:G) = \pitchfork(H_2:G)$.
    \item $\mathcal{HBI}(\Sigma(H_1):\Sigma)^{\Theta} = \mathcal{HBI}(\Sigma(H_2):\Sigma)^{\Theta}$.
\end{enumerate}
\end{enumerate}
\end{theorem}

We will discuss that Kobayashi's properness criterion (Theorem \ref{theorem:Kobayashi96})  can be explained by combining Theorem \ref{theorem:main_section} with some facts for metric geometry on Riemannian symmetric spaces without compact factors (see Section \ref{section:KobayashiCriterion}).

\begin{remark}\label{remark:YBBF}
Let us fix $n \in \Z_{\geq 0}$ 
and a closed Lie subgroup $K$ of the orthogonal group $O(n)$.
We define the Lie group $G := K \ltimes \R^n$ as the semi-direct product of $K$ and the additive Lie group $\R^n$.
Then $G$ has the proper isomeric transitive action on the $n$-dimensional Euclidean space $M := \R^n$,
and one can easily to check that 
the $G$-space $M$ has Property (S). 
Therefore, Theorem \ref{theorem:main} (more precisely, Theorem \ref{theorem:mainlong} stated in Section \ref{section:proof}) works for 
$G := K \ltimes \R^n$.
It should be remarked that 
the properness criterion 
for pairs of connected or discrete subgroups of $G$
given by Baklouti--Bejar--Fendri \cite[Theorem 3.9]{BakloutiBejarFendri2023CptExtRn}
is related to our works in the sense above.
Furthermore, if we take a cross-section $\Sigma$ of $\R^n$ for the $K$-representation on $\R^n$, 
then Theorem \ref{theorem:main_section} also works.
This contains Yoshino's properness criterion 
 \cite[Theorem 9]{Yoshino2007CartanMotion} for the pair of subsets of a Cartan motion group $G = K \ltimes \R^n$ associated to a real linear reductive group.
\end{remark}

\begin{remark}\label{remark:BenoistBuilding}
As we mentioned in Section \ref{section:introduction}, 
similarly to Theorem \ref{theorem:Kobayashi96} \eqref{item:Kobayashi96:1}, 
Benoist \cite{Benoist96} gave a properness criterion
for pairs of closed subgroups of 
a simply-connected semisimple algebraic group $G$ over a non-Archimedean local field $k$. 
We believe that his criterion can be explained by applying Theorem \ref{theorem:main_section} to the isometric $G$-action on the Bruhat--Tits building $\mathcal{B}(G;k)$ of $G$ and an apartment $\Sigma$ of $\mathcal{B}(G;k)$.
This will be explored in detail
 in a future work.
\end{remark}

\begin{remark}
Gu\'{e}ritaud--Kassel \cite[Theorem 1.8]{GueritaudKassel2017maximally} 
proved that a discontinuous group for 
the group manifold $PO(n,1)$ 
should be ``sharp'' (see \cite{KasselKobayashi16} for the definition of sharpness of discontinuous groups)
under a suitable assumption.
Their method is based on metric geometry on the hyperbolic spaces.
We are also interested in sharpness of discontinuous groups 
and in a future work, we will investigate such problems in our setting.
\end{remark}

\section{Non-expanding proper surjective maps and HBI-pairs}\label{section:nonexpanding}

Let $(\Omega,d)$ and $(\Omega',d')$ be both proper metric spaces.
In this section, we consider a continuous proper surjective map $\eta : \Omega \rightarrow \Omega'$ satisfying the following two conditions:
\begin{description}
    \item[Condition (1)] The map $\eta : \Omega \rightarrow \Omega'$ is non-expanding, that is, 
    \[
    d'(\eta(x),\eta(y)) \leq d(x,y)
    \]
    holds for any $x,y \in \Omega$.
    \item[Condition (2)] For each $x \in \Omega$ and $y' \in \Omega'$, 
    there exists $y \in \eta^{-1}(y')$ such that 
    \[
    d(x,y) = d'(\eta(x),y').
    \]
\end{description}

For each subset $C$ of $\Omega$, we define the subset $\Theta(C)$ of $\Omega$ by 
\[
\Theta(C) := \eta^{-1} (\eta(C)) \subset \Omega.
\]
Then $\Theta$ gives a transformation on the set $\mathcal{P}^\times(\Omega)$ of all non-empty subsets of $\Omega$.
Furthermore, one can see that 
$\Theta^2 = \Theta$ and $\Theta$ preserves inclusions.

In this section, 
we give a relationship between some types of HBI-pairs in $\Omega$ and those in $\Omega'$ via $\eta$ (see Theorem \ref{theorem:OmOm}),
and such the theorem can be applied for 
spaces of compact group orbits in proper metric spaces (see Section \ref{subsection:Korbits}).

\subsection{A relationship between HBI-pairs via non-expanding maps}

The purpose of this subsection 
is to give a proof of the following theorem, which gives a relationship between some types of HBI-pairs in $\Omega$ and those in $\Omega'$ via $\eta$:

\begin{theorem}\label{theorem:OmOm}
Let $C_1$ and $C_2$ be both non-empty subsets of $\Omega$.
\begin{enumerate}[(1)]
    \item \label{item:OmOm:CThetaC} If $C_1 \sim C_2$ in $\Omega$, then $\Theta(C_1) \sim \Theta(C_2)$ in $\Omega$.
    \item \label{item:OmOm:ThetaCthetaC} The following two conditions on $(C_1,C_2)$ are equivalent:
    \begin{enumerate}[(\textrm{2}a)]
        \item $\Theta(C_1) \sim \Theta(C_2)$ in $\Omega$.
        \item $\eta(C_1) \sim \eta(C_2)$ in $\Omega'$.
    \end{enumerate}
    \item \label{item:OmOm:HBI} The following four conditions on $(C_1,C_2)$ are equivalent:    
\begin{enumerate}[(\textrm{3}a)]
        \item \label{item:OmOm:C1KC2} The pair $(C_1, \Theta(C_2))$ is HBI in $\Omega$.
        \item \label{item:OmOm:KC1C2} The pair $(\Theta(C_1), C_2)$ is HBI in $\Omega$.
        \item \label{item:OmOm:KK} The pair $(\Theta(C_1), \Theta(C_2))$ is HBI in $\Omega$.
        \item \label{item:OmOm:mumu} The pair $(\eta(C_1), \eta(C_2))$ is HBI in $\Omega'$.
\end{enumerate}
\end{enumerate}    
\end{theorem}

Let us give a proof of  
Theorem \ref{theorem:OmOm}, 
by applying the lemma below:

\begin{lemma}\label{lemma:dd_property}
Let $C$ be a non-empty subset of $\Omega$
and 
$C'$ a non-empty subset of $\Omega'$.
We also fix $r \geq 0$.
Then the following equations hold:
\begin{align}
\eta^{-1}(\overline{N}_r(C')) &= \overline{N}_r(\eta^{-1}(C')) \quad \text{ in } \Omega. \label{eq:Cp} \\
\eta(\overline{N}_r(C)) &= \overline{N}_r(\eta(C))  \quad \text{ in } \Omega', \label{eq:C} \\
\Theta (\overline{N}_r(C))) &= \overline{N}_r(\Theta(C))) \quad \text{ in } \Omega. \label{eq:ThetaC}
\end{align}
\end{lemma}

\begin{proof}[Proof of Lemma \ref{lemma:dd_property}]
The equations \eqref{eq:Cp} and \eqref{eq:C} comes immediately from Conditions (1) and (2) on the map $\eta : \Omega \rightarrow \Omega'$.
The equality \eqref{eq:ThetaC} is obtained by the equations \eqref{eq:Cp} and \eqref{eq:C}.
\end{proof}

\begin{proof}[Proof of Theorem \ref{theorem:OmOm}]
First, the claim \eqref{item:OmOm:CThetaC} comes immediately from the equation \eqref{eq:ThetaC} and the observation that $\Theta$ preserves inclusions.
Second, the claim \eqref{item:OmOm:ThetaCthetaC} can be obtained by the equations \eqref{eq:Cp} and \eqref{eq:C} in Lemma \ref{lemma:dd_property}.
We shall prove the claim \eqref{item:OmOm:HBI}.
The implications 
\eqref{item:OmOm:KK} $\Rightarrow$ \eqref{item:OmOm:C1KC2} 
and 
\eqref{item:OmOm:KK} $\Rightarrow$ \eqref{item:OmOm:KC1C2}
are both trivial.
Let us prove the implications 
\eqref{item:OmOm:C1KC2} $\Rightarrow$ \eqref{item:OmOm:KK} and 
\eqref{item:OmOm:KC1C2} $\Rightarrow$ \eqref{item:OmOm:KK}.
Fix any $r \geq 0$.
Then by Condition (2) on $\eta$, one can easily see that 
the following two equations hold:
\begin{align*}
\Theta(C_1) \cap \overline{N}_r(\Theta(C_2))
    &= \Theta (C_1 \cap \overline{N}_r(\Theta(C_2))), \\
\overline{N}_r(\Theta(C_1)) \cap \Theta(C_2)
    &= \Theta(\overline{N}_r(\Theta(C_1)) \cap C_2). 
\end{align*}
Both of the implications 
\eqref{item:OmOm:C1KC2} $\Rightarrow$ \eqref{item:OmOm:KK} and 
\eqref{item:OmOm:KC1C2} $\Rightarrow$ \eqref{item:OmOm:KK}
come immediately from the observation 
that for any bounded subset $C_0$ of $\Omega$, 
the subset $\Theta(C_0)$ is also a bounded in $\Omega$, 
and equations above.
Finally, let us prove the equivalence \eqref{item:OmOm:KK} $\Leftrightarrow$ \eqref{item:OmOm:mumu}.
Fix $r \geq 0$.
Then by Lemma \ref{lemma:dd_property}, we have:
\begin{align*}
    \eta(C_1) \cap \overline{N}_r(\eta(C_2))
        &= \eta(C_1) \cap \eta(\overline{N}_r (C_2)) \\
        &= \eta(\Theta(C_1)) \cap \eta(\Theta(\overline{N}_r(C_2))) \\
        &= \eta(\Theta(C_1) \cap \Theta(\overline{N}_r(C_2))) \\
        &= \eta(\Theta(C_1) \cap \overline{N}_r(\Theta(C_2))).
\end{align*}
Since $\Theta(C_1) \cap \overline{N}_r(\Theta(C_2))$ is a $\Theta$-fixed subset of $\Omega$ and $\eta$ is a proper map, 
one can see that 
$\eta(C_1) \cap \overline{N}_r(\eta(C_2))$ is bounded in $\Omega'$ if and only if $\Theta(C_1) \cap \overline{N}_r(\Theta(C_2))$ is bounded in $\Omega$.
This proves the equivalence
\eqref{item:OmOm:KK} $\Leftrightarrow$ \eqref{item:OmOm:mumu}.
\end{proof}

We shall put
\begin{align*}
\mathcal{P}^\times(\Omega)^{\Theta} &:= \{ C \in \mathcal{P}^{\times}(\Omega) \mid \Theta(C) = C \}, \\
[\mathcal{P}^\times(\Omega)^{\Theta}] &:= \{ [C] \in [\mathcal{P}^\times(\Omega)] \mid C \in \mathcal{P}^\times(\Omega)^{\Theta} \} \subset [\mathcal{P}^\times(\Omega)].
\end{align*}
Furthermore, for each non-empty subset $C$ of $\Omega$,  we also write 
\begin{align*}
\mathcal{HBI}(C:\Omega)^{\Theta} 
&:= \mathcal{HBI}(C:\Omega) \cap \mathcal{P}^\times(\Omega)^{\Theta} \subset \mathcal{HBI}(C:\Omega), \\
\mathcal{HBI}([C]:\Omega)^{\Theta} 
    &:= \mathcal{HBI}([C]:\Omega) \cap [\mathcal{P}^\times(\Omega)^{\Theta}] \subset \mathcal{HBI}([C]:\Omega).    
\end{align*}

Then, as a corollary to Theorem \ref{theorem:OmOm}, 
we also obtain the following:

\begin{corollary}\label{corollary:bijTheta}
Let us consider the bijective map 
\[
\eta_* : \mathcal{P}^\times(\Omega)^{\Theta} \stackrel{\sim}{\longrightarrow} \mathcal{P}^\times(\Omega'),~ C \mapsto \eta(C).
\]
\begin{enumerate}[(1)]
\item 
For each $C \in \mathcal{P}^\times(\Omega)^{\Theta}$, 
the map $\eta_*$ gives a one-to-one correspondence between $\mathcal{HBI}(C:\Omega)^{\Theta}$ and $\mathcal{HBI}(\eta(C):\Omega')$.
\item The bijection $\eta_* : \mathcal{P}^\times(\Omega)^{\Theta} \stackrel{\sim}{\rightarrow} \mathcal{P}^\times(\Omega')$ induces the bijective map below
\[
\eta_*^{-} : [\mathcal{P}^\times(\Omega)^{\Theta}] \stackrel{\sim}{\longrightarrow} [\mathcal{P}^\times(\Omega')],~ [C] \mapsto [\eta(C)].
\]
Furthermore, 
for each $C \in \mathcal{P}^\times(\Omega)^{\Theta}$, 
the map $\eta_*^{-}$ gives a one-to-one correspondence between 
$\mathcal{HBI}([C]:\Omega)^{\Theta}$ 
and 
$\mathcal{HBI}([\eta(C)]:\Omega')$.
\end{enumerate}
\end{corollary}

\subsection{Spaces of orbits in proper metric spaces}\label{subsection:Korbits}

Let $(M,d)$ be a proper metric space equipped with a  continuous isometric action of a compact group $K$.
We denote by $K \backslash M$ the space of $K$-orbits in $M$ and by 
\[
\varpi : M \rightarrow K \backslash M, ~ p \mapsto K \cdot p
\]
the quotient map. 
Note that the space $K \backslash M$ is locally-compact Hausdorff with respect to the quotient topology,
and the map $\varpi : M \rightarrow K \backslash M$ is proper.
Then     
\begin{align*}
    d^K : K \backslash M &\times K \backslash M \rightarrow \R_{\geq 0}, \\
    (\mathcal{O}_1,&\mathcal{O}_2) \mapsto \min\{ d(x,y)  \mid (x,y) \in \mathcal{O}_1 \times \mathcal{O}_2 \}
    \end{align*}
is well-defined 
and defines a proper metric on the space $K \backslash M$.
The quotient topology on $K \backslash M$ coincides with the topology on $K \backslash M$ induced by the metric $d^K$. 

One can easily check that 
the map $\varpi : M \rightarrow K \backslash M$ satisfies Conditions (1) and (2),
and thus 
Theorem \ref{theorem:OmOm} and Corollary \ref{corollary:bijTheta} can be applied for $(\Omega,\Omega',\eta) = (M,K\backslash M,\varpi)$.

Furthermore, let $\Sigma$ be a closed subset of $M$ satisfying the following two conditions:
\begin{itemize}
    \item For each $K$-orbit $\mathcal{O}$ in $M$, the intersection $\Sigma \cap \mathcal{O}$ is non-empty.
    \item For each pair of $K$-orbits $(\mathcal{O}_1,\mathcal{O}_2)$ in $M$, 
    there exists $(p,q) \in (\Sigma \cap \mathcal{O}_1) \times (\Sigma \cap \mathcal{O}_2)$ such that $d(p,q) = d^K(\mathcal{O}_1,\mathcal{O}_2)$.
\end{itemize}
Let us consider $\Sigma$ as a proper metric space 
with respect to the restriction of the metric $d$ on $\Sigma$.
Then the restriction 
\[
\eta := \varpi|_{\Sigma} : \Sigma \rightarrow K \backslash M, ~ p \mapsto K \cdot p
\] 
is surjective and satisfies Conditions (1) and (2).
Therefore,  
Theorem \ref{theorem:OmOm} and Corollary \ref{corollary:bijTheta} can be applied for $(\Omega,\Omega',\eta) = (\Sigma,K\backslash M,\eta)$.

Let us consider the situation that 
$M$ is a connected complete Riemannian manifold 
and the $K$-action on $M$ is smooth and isometric in the sense that the Riemannian metric tensor on $M$ is fixed by $K$. 
A connected, closed, regular smooth submanifold $\Sigma$ of $M$ is called a section for the $K$-action on $M$ if it meets all $K$-orbits in $M$ orthogonally (see \cite{PT87} for the details).
Furthermore, the $K$-action on $M$ is said to be polar [resp. hyperpolar] if it admits a section [resp. flat section].

A section $\Sigma$ for the $K$-action on $M$ satisfies the two conditions above (a proof can be found in \cite[Proposition 1.3.2]{Gorodski2022arXivTopics}),
and thus even in the situation above, 
Theorem \ref{theorem:OmOm} and Corollary \ref{corollary:bijTheta} can be applied for $(\Omega,\Omega',\eta) = (\Sigma,K\backslash M,\eta)$.
Note that for a section $\Sigma$ for the $K$-action  on $M$, 
by defining
\begin{align*}
N_K(\Sigma) &:= \{ k \in K \mid k \cdot \Sigma = \Sigma \} \subset K, \\
Z_K(\Sigma) &:= \{ k \in K \mid k \cdot p = p \text{ for any } p \in \Sigma \} \subset N_K(\Sigma),
\end{align*}
we obtain the generalized Weyl group $W := N_K(\Sigma)/Z_K(\Sigma)$ acting on the section $\Sigma$.
It is known that the generalized Weyl group $W$ is finite (in our setting), 
and for each $p \in \Sigma$, 
\[
(K \cdot p) \cap \Sigma = W \cdot p
\]
(see \cite[Section 4]{PT87} for the details).
In particular, in this situation, we have 
\[
\Theta(C) := \eta^{-1}(\Theta(C)) = W \cdot C 
\]
for any non-empty subset $C$ of $\Sigma$.

\section{Proofs of Theorems \ref{theorem:main} and \ref{theorem:main_section}}\label{section:proof}

The goal of this section is to give proofs of  
Theorems \ref{theorem:main} and \ref{theorem:main_section} stated in Section \ref{subsection:main}.

Let $(M,d)$ be a proper metric space 
equipped with a continuous proper isometric transitive action of 
a locally-compact group $G$.
We fix a base point $*$ of $M$, 
and put 
\[
\pi : G \rightarrow M, ~ g \mapsto g \cdot *.
\]
Note that the map $\pi$ is continuous proper and surjective.
The compact isotropy subgroup of $G$ at the point $*$ is denoted by $K$. 

We first remark that, in the situation above, 
by Propositions \ref{proposition:LHHBIdual}, 
$G$ should be $\sigma$-compact.

Let us introduce the following notions:
For a non-empty subset $H$ of $G$ 
and a non-empty compact subset $C$ of $M$, 
we define the non-empty subset $S^H(C)$ of $M$ by 
\begin{align*}
S^H(C) 
    &:= \pi^{-1}(C) \cdot \pi(H) \\
    &= \{ g \cdot \pi(h) \mid g \in \pi^{-1}(C) \text{ and } h \in H \} \subset M.
\end{align*}
For the case where $C = \{ * \}$, 
we just put 
\[
S^H(*) := S^H(\{* \}) = K \cdot \pi(H) = K \cdot H \cdot * \subset M.
\]

\subsection{Proper isometric transitive actions on metric spaces and proper pairs}

In this subsection, our goal is to show the following theorem:

\begin{theorem}\label{theorem:GM}
Let $H_1$ and $H_2$ be both non-empty subsets of $G$.
\begin{enumerate}
    \item \label{item:GM:SNSD}  If $S^{H_1}(*) \sim S^{H_2}(*)$ in $M$, then $H_1 \sim H_2$ in $G$.
    \item \label{item:GM:PPGimplCCM} If the pair $(H_1,H_2)$ is proper in $G$, 
    then the pair $(S^{H_1}(*),S^{H_2}(*))$ is HBI in $M$.
\end{enumerate}    
\end{theorem}

To show Theorem \ref{theorem:GM}, 
let us give a proof of the lemma below:

\begin{lemma}\label{lemma:piinvNS}
Let $H$ be a non-empty subset of $G$ and fix $r \geq 0$.
We put $D_r := \pi^{-1}(\overline{N}_r(*)) \subset G$.
Then 
\[
\pi^{-1}(\overline{N}_r(S^{H}(*))) = K \cdot H \cdot D_r.
\]
\end{lemma}

\begin{proof}[Proof of Lemma \ref{lemma:piinvNS}]
First, we show 
\[
\pi^{-1}(\overline{N}_r(S^{H}(*))) \subset K \cdot H \cdot D_r.
\]
Take any $g \in \pi^{-1}(\overline{N}_r(S^{H}(*)))$.
Put $p := \pi(g) \in \overline{N}_r(S^{H}(*))$.
One can find and do take $q \in S^{H}(*)$ with $d(p,q) \leq r$.
Since $S^{H}(*) = K \cdot H \cdot *$, 
one can also take $k \in K$ and $h \in H$ with $q = k \cdot h \cdot *$. 
Then 
\begin{align*}
    d(\pi(h^{-1} \cdot k^{-1} \cdot g),*)
        &= d(p,q) \leq r.
\end{align*}
Therefore, $h^{-1} \cdot k^{-1} \cdot g \in D_r$ 
and thus $g \in K \cdot H \cdot D_r$.

Conversely, let us prove 
\[
\pi^{-1}(\overline{N}_r(S^{H}(*))) \supset K \cdot H \cdot D_r.
\]
Fix any $k \in K$, $h \in H$ and $d \in D_r$.
We put $p := \pi(d) \in \overline{N}_r(*)$.
Then $\pi(k \cdot h \cdot d) = k \cdot h \cdot p$
and 
\[
d(k \cdot h \cdot p,k \cdot h \cdot *) = d(p,*) \leq r.
\]
Since $k \cdot h \cdot * \in S^{H}(*)$, 
we have that $\pi(k \cdot h \cdot d) \in \overline{N}_r(S^H(*))$.
\end{proof}

We give a proof of Theorem \ref{theorem:GM} as below:

\begin{proof}[Proof of Theorem \ref{theorem:GM}]
First, we shall prove the claim \eqref{item:GM:SNSD}.
Suppose that 
\[
S^{H_1}(*) \subset \overline{N}_r(S^{H_2}(*)) \text{ and } S^{H_2}(*) \subset \overline{N}_r(S^{H_1}(*))
\]
for $r \geq 0$.
Put $D_r := \pi^{-1}(\overline{N}_r(*))$.
Then $D_r$ is compact 
and $D_r^{-1} = D_r \supset K$.
For $(i,j) = (1,2)$ or $(2,1)$, 
by Lemma \ref{lemma:piinvNS}, 
we have 
\begin{align*}
H_i &\subset K \cdot H_i \cdot K \\
    &= \pi^{-1}(S^{H_i}(*)) \\
    &\subset \pi^{-1}(\overline{N}_r(S^{H_j}(*))) \\
    &= K \cdot H_j \cdot D_r \\
    &\subset D_r \cdot H_j \cdot D_r^{-1}.
\end{align*}
This proves $H_1 \sim H_2$.

Next, let us prove the claim \eqref{item:GM:PPGimplCCM}.
Assume that the pair $(H_1,H_2)$ is proper in $G$.
Let us fix $r \geq 0$.
By Proposition \ref{proposition:CI}, 
we only need to show that the subset $S^{H_1}(*) \cap \overline{N}_r(S^{H_2}(*))$ of $M$ is bounded.    
Put $D_r := \pi^{-1}(\overline{N}_r(*))$.
Then by Lemma \ref{lemma:piinvNS}, we have 
    \begin{align*}
    \pi^{-1}(S^{H_1}(*) \cap \overline{N}_r(S^{H_2}(*)))
        &= \pi^{-1}(S^{H_1}(*)) \cap \pi^{-1}(\overline{N}_r(S^{H_2}(*))) \\
        &= (K \cdot H_1 \cdot K) \cap (K \cdot H_2 \cdot D_r).
    \end{align*}
Since $K \cdot H_1 \cdot K \in [H_1]$, $K \cdot H_2 \cdot D_r \in [H_2]$ and $([H_1],[H_2])$ is proper in $G$, the intersection $(K \cdot H_1 \cdot K) \cap (K \cdot H_2 \cdot D_r)$ is relatively compact in $G$.
Thus $S^{H_1}(*) \cap \overline{N}_r(S^{H_2}(*))$ is bounded in $M$ since $\pi$ is surjective.
\end{proof}

\subsection{Proofs of Theorems \ref{theorem:main} and \ref{theorem:main_section}}\label{subsection:proofThmmain}

In this subsection, we give proofs of Theorems \ref{theorem:main} and \ref{theorem:main_section}.
We suppose that our metric $G$-space $M$ has the property (S) stated in Section \ref{section:introduction}.

Then the following theorem, which claims the converse implications of those in Theorem \ref{theorem:GM}, holds:

\begin{theorem}\label{theorem:GMstar}
In the setting above, 
let $H_1$ and $H_2$ be both non-empty subsets of $G$.
\begin{enumerate}
    \item \label{item:GM:SNSDstar}  $H_1 \sim H_2$ in $G$ if and only if $S^{H_1}(*) \sim S^{H_2}(*)$ in $M$.
    \item \label{item:GM:PPGimplCCMstar} The pair $(H_1,H_2)$ is proper in $G$ if and only if 
    the pair $(S^{H_1}(*),S^{H_2}(*))$ is HBI in $M$.
\end{enumerate}    
\end{theorem}

The following lemma plays a key role in our proof of 
Theorem \ref{theorem:GMstar}.

\begin{lemma}\label{lemma:SNNS}
Let $H$ be a non-empty closed subset of $G$.
Then for each $r \geq 0$, there exists $\tau \geq 0$ such that 
\[
S^{H}(\overline{N}_r(*)) \subset \overline{N}_{\tau}(S^{H}(*)).
\]
\end{lemma}

\begin{proof}[Proof of Lemma \ref{lemma:SNNS}]
Let us fix $\tau = \tau(r) \geq 0$ as in Property (S).
Take any $q \in S^{H}(\overline{N}_r(*))$, 
and we shall prove $q \in \overline{N}_{\tau}(S^{H}(*))$.
Fix $g \in \pi^{-1}(\overline{N}_r(*))$ and $h \in H$ with $q = g \cdot \pi(h)$.
We put $p := \pi(g) \in \overline{N}_r(*)$.
Then by the definition of $\tau$,
there exists $g' \in G$ such that $g' \cdot p = *$ and $d(q,g' \cdot q) \leq \tau$.
We put $q_0 := g' \cdot q \in M$.
Since 
\begin{align*}
\begin{cases}
    (g')^{-1} \cdot * &= p = g \cdot *,  \\
    (g')^{-1} \cdot q_0 &= q = g \cdot h \cdot *, 
\end{cases}
\end{align*}
we obtain $k := g' g \in K$ and $q_0 = k \cdot h \cdot * \in S^{H}(*)$.
This proves that $q \in \overline{N}_{\tau}(S^{H}(*))$.
\end{proof}

Let us give a proof of Theorem \ref{theorem:GMstar} as below.

\begin{proof}[Proof of Theorem \ref{theorem:GMstar}]
First, we discuss the claim \eqref{item:GM:SNSDstar} in 
Theorem \ref{theorem:GMstar}.
The ``if''-part of the claim is just Theorem \ref{theorem:GM} \eqref{item:GM:SNSD}.
We show the ``only if''-part.
Assume $H_1 \sim H_2$ in $G$ and we shall prove $S^{H_1}(*) \sim S^{H_2}(*)$ in $M$.
Fix a compact subset $D$ of $G$ with $H_1 \subset \overline{N}_D(H_2)$ and $H_2 \subset \overline{N}_{D}(H_1)$.
We define the subset $D'$ of $G$ by $D' := D \cdot K$.
Since $D^{-1}$ and $D'$ are both compact, 
one can find and do take $r_1,r_2 \geq 0$ with 
\[
\pi(D^{-1}) \subset \overline{N}_{r_1}(*) \text{ and }
\pi(D') \subset \overline{N}_{r_2}(*).
\]
Take $\tau \geq 0$ as in Lemma \ref{lemma:SNNS} for $r = r_2$, and put $r' := r_1 + \tau \geq 0$.
Then for $(i,j) = (1,2)$ or $(2,1)$, we obtain 
\begin{align*}
    S^{H_i}(*) 
        &= K \cdot \pi(H_i) \\
        &\subset K \cdot \pi(N_D(H_j)) \\
        &= K \cdot \pi(D \cdot H_j \cdot D^{-1}) \\
        &= D' \cdot H_j \cdot \pi(D^{-1}) \\
        &\subset D' \cdot H_j \cdot \overline{N}_{r_1}(*) \\
        &= \overline{N}_{r_1}(D' \cdot H_j \cdot *) \\
        &\subset \overline{N}_{r_1}(\pi^{-1}(\overline{N}_{r_2}(*)) \cdot \pi(H_j)) \\
        &= \overline{N}_{r_1}(S^{H_j}(\overline{N}_{r_2}(*))) \quad (\because \text{ the definition of } S^{H_j}(\overline{N}_{r_2}(*)))\\
        &\subset \overline{N}_{r_1} \overline{N}_{\tau} (S^{H_j}(*)) \quad (\because \text{ Lemma } \ref{lemma:SNNS})\\
        &\subset \overline{N}_{r'} (S^{H_j}(*)).
\end{align*}
This proves $S^{H_1}(*) \sim S^{H_2}(*)$ in $M$. 

Next, we shall give a proof of the claim \eqref{item:GM:PPGimplCCMstar} in Theorem \ref{theorem:GMstar}.
The ``only if''-part of the claim is just Theorem \ref{theorem:GM} \eqref{item:GM:PPGimplCCM}.
Let us prove the ``if''-part.
Assume that 
the pair $(S^{H_1}(*),S^{H_2}(*))$ is HBI in $M$.
Take any $H_1' \in [H_1]$ and any $H_2' \in [H_2]$.
We shall prove that the intersection $H_1' \cap H_2'$ is relatively compact in $G$.
Since $\pi : G \rightarrow M$ is a continuous proper surjective map and $(M,d)$ is a proper metric space, we only need to show that 
the image $\pi(H_1' \cap H_2')$ is bounded in $M$.
Since 
\begin{align*}
    \pi(H_1' \cap H_2') 
        &\subset \pi(H_1') \cap \pi(H_2') \\
        &\subset (K \cdot \pi(H_1')) \cap (K \cdot \pi(H_2')) \\
        &= S^{H_1'}(*) \cap S^{H_2'}(*), 
\end{align*}
it is enough to show that 
the intersection $S^{H_1'}(*) \cap S^{H_2'}(*)$ is bounded in $M$.
Recall that the claim \eqref{item:GM:SNSDstar} in Theorem \ref{theorem:GMstar} has been proved above.
Thus, we have $S^{H_i'}(*) \in [S^{H_i}(*)]$ for $i = 1,2$.
Therefore, by the assumption, 
the intersection $S^{H_1'}(*) \cap S^{H_2'}(*)$ is bounded in $M$.
\end{proof}

As in Section \ref{subsection:Korbits},
we consider the proper metric space $(K \backslash M,d^K)$,
the quotient map  
\[
\varpi : M \rightarrow K \backslash M,
\]
and define 
\[
\mu := \varpi \circ \pi : G \rightarrow K \backslash M, ~ g \mapsto K \cdot (g \cdot *)
\]
Note that, as we mentioned in Section \ref{subsection:Korbits}, 
Theorem \ref{theorem:OmOm} and 
Corollary \ref{corollary:bijTheta} 
can be applied for $\varpi : M \rightarrow K \backslash M$.

To prove Theorem \ref{theorem:main}, 
it is enough to show the following theorem:

\begin{theorem}\label{theorem:mainlong}
In the setting above, 
let $H_1$ and $H_2$ be both non-empty subsets of $G$.
\begin{enumerate}[(1)]
    \item \label{item:main1} The following five conditions on $(H_1,H_2)$ are equivalent:
    \begin{enumerate}[(\textrm{1}i)]
    \item \label{item:main:pitch:G} The pair $(H_1,H_2)$ is proper in $G$.
    \item \label{item:main:pitch:M} The pair $(S^{H_1}(*), S^{H_2}(*))$ is HBI in $M$.
    \item \label{item:main:pitch:M1} The pair $(\pi(H_1), S^{H_2}(*))$ is HBI in $M$.
    \item \label{item:main:pitch:M2} The pair $(S^{H_1}(*), \pi(H_2))$ is HBI in $M$.
    \item \label{item:main:pitch:KM} The pair $(\mu(H_1), \mu(H_2))$ is HBI in $K \backslash M$.
    \end{enumerate}
\item \label{item:main2} The following six conditions on $(H_1,H_2)$ are equivalent: 
\begin{enumerate}[(\textrm{2}i)]
    \item \label{item:main:sim:G} $H_1 \sim H_2$ in $G$.
    \item \label{item:main:sim:M} $S^{H_1}(*) \sim S^{H_2}(*)$ in $M$.
    \item \label{item:main:sim:KM} $\mu(H_1) \sim \mu(H_2)$ in $K \backslash M$.
    \item \label{item:main:sim:PPG} $\pitchfork(H_1:G) =  \pitchfork(H_2:G)$.
    \item \label{item:main:sim:CCM} $\mathcal{HBI}(S^{H_1}(*):M) = \mathcal{HBI}(S^{H_2}(*):M)$.
    \item \label{item:main:sim:CCKM} $\mathcal{HBI}(\mu(H_1):K \backslash M) = \mathcal{HBI}(\mu(H_2):K \backslash M)$. 
\end{enumerate}
\end{enumerate}    
\end{theorem}

\begin{proof}[Proof of Theorem \ref{theorem:mainlong}]
Our strategy is to prove the following equivalences:
\[
\xymatrix{& \eqref{item:main:pitch:G}  \ar@{<=>}[d] &
 \\ \eqref{item:main:pitch:M1} & \eqref{item:main:pitch:M} \ar@{<=>}[d] \ar@{<=>}[r] \ar@{<=>}[l] & \eqref{item:main:pitch:M2} \\ & \eqref{item:main:pitch:KM} &\\} \quad \quad 
 \xymatrix{\eqref{item:main:sim:G} \ar@{<=>}[d] \ar@{<=>}[r] & \eqref{item:main:sim:PPG} 
 \\ \eqref{item:main:sim:M} \ar@{<=>}[r] \ar@{<=>}[d] & \eqref{item:main:sim:CCM}
 \\ \eqref{item:main:sim:KM} \ar@{<=>}[r] & \eqref{item:main:sim:CCKM}
}
\]

First, both of the equivalences \eqref{item:main:sim:M} $\Leftrightarrow$
\eqref{item:main:sim:CCM} and \eqref{item:main:sim:KM} $\Leftrightarrow$
\eqref{item:main:sim:CCKM} come from Proposition \ref{proposition:CCdual}. 
Furthermore, 
the equivalence \eqref{item:main:sim:G} $\Leftrightarrow$ \eqref{item:main:sim:PPG} is followed by Proposition \ref{proposition:LHHBIdual}.
Next, the equivalence \eqref{item:main:sim:M}
$\Leftrightarrow$ \eqref{item:main:sim:KM} 
can be proved by applying Theorem \ref{theorem:OmOm} \eqref{item:OmOm:ThetaCthetaC} for  
\[
(\Omega,\Omega',\eta,C_1,C_2) = 
(M, K \backslash M, \varpi, \pi(H_1),\pi(H_2))
\]
since $\mu(H) = \varpi(S^{H}(*))$ for each non-empty subset $H$ of $G$.
By the similar arguments, 
the equivalences among the conditions 
\eqref{item:main:pitch:M}, 
\eqref{item:main:pitch:M1},  
\eqref{item:main:pitch:M2} and 
\eqref{item:main:pitch:KM} are followed by Theorem \ref{theorem:OmOm} \eqref{item:OmOm:HBI}.
Both of the equivalences 
\eqref{item:main:pitch:G} $\Leftrightarrow$ \eqref{item:main:pitch:M} 
and 
\eqref{item:main:sim:M} $\Leftrightarrow$ \eqref{item:main:sim:G} 
are already proved as Theorem \ref{theorem:GMstar}.
\end{proof}

Let us consider the setting in Theorem \ref{theorem:main_section}.
Note that as we mentioned in Section \ref{subsection:Korbits}, 
Theorem \ref{theorem:OmOm} and 
Corollary \ref{corollary:bijTheta} 
can be applied to the map $\eta : \Sigma \rightarrow K \backslash M$.

We give a proof of Theorem \ref{theorem:main_section} by applying the following observations below:

\begin{lemma}\label{lemma:Ximu}
For a non-empty subset $H$ of $G$, 
the two equations below holds:
\begin{align*}
\Sigma(H) &= \eta^{-1}(\mu(H)) \text{ in } \Sigma, \\
\mu(H) &= \eta(\Sigma(H)) \text{ in } K \backslash M.
\end{align*}
\end{lemma}

\begin{proof}[Proof of Theorem \ref{theorem:main_section}]
Theorem \ref{theorem:main_section}
comes immediately from 
Theorem \ref{theorem:mainlong},  
Lemma \ref{lemma:Ximu}
and 
applying 
Theorem \ref{theorem:OmOm} \eqref{item:OmOm:HBI}
and 
Corollary \ref{corollary:bijTheta}
for 
 \[
 (\Omega,\Omega',\eta,C_1,C_2) = (\Sigma,K \backslash M, \eta, \Sigma(H_1),\Sigma(H_2)).
 \]
\end{proof}

\section{CAT(0) spaces and Property (S)}\label{section:sufficientP}

Let $G$ be a group and $M$ a proper metric space equipped with an isometric transitive $G$-action.
In this section, in terms of geometry on CAT(0) spaces, 
we give a sufficient condition for 
property (S), stated in Section \ref{section:introduction}, of the metric $G$-space $M$ (see Theorem \ref{theorem:SinCATzero} and Corollary \ref{corollary:Ptransitive}).

\subsection{Notations for CAT(0) spaces}\label{subsection:notationCATzero}

In this subsection, 
following \cite{BridsonHaefligerMetricSpaces} and \cite{Caprace2014Lectures}, 
we fix our terminologies for CAT(0) spaces.

Let $(\Omega,d)$ be a metric space.
For $-\infty < a < b < \infty$, 
an isometric map $c : [a,b] \rightarrow M$ is called a geodesic segment joining the points $c(a)$ and $c(b)$ in $\Omega$.
A metric space $(\Omega,d)$ is said to be a geodesic metric space 
if any two points of $\Omega$ can be joined by a geodesic segment in $\Omega$.

Let $(\Omega,d)$ be a geodesic metric space.
For each triple $(x_1,x_2,x_3) \in \Omega \times \Omega \times \Omega$, 
a triple $(x_1',x_2',x_3') \in \R^2 \times \R^2 \times \R^2$ 
is called a Euclidean comparison triangle 
for $(x_1,x_2,x_3)$ if $d(x_i,x_j) = d_{\R^2}(x_i',x_j')$ for any $i,j = 1,2,3$,
where $d_{\R^2}$ denotes the standard metric on the Euclidean plane $\R^2$.
Note that for any triple $(x_1,x_2,x_3)$ of points in $\Omega$, 
a Euclidean comparison triangle for $(x_1,x_2,x_3)$ in $\R^2$ exists uniquely up to congruence.
A geodesic metric space $(\Omega,d)$ is called CAT(0) if 
for 
any triple $(x_1,x_2,x_3) \in \Omega \times \Omega \times \Omega$,
any point $p$ on any geodesic segment joining $x_2$ and $x_3$, 
and any Euclidean comparison triangle $(x_1',x_2',x_3')$ for $(x_1,x_2,x_3)$, 
the inequality below holds:
\[
d(x_1,p) \leq d_{\R^2}(x_1',p')
\]
where $p'$ denotes the unique point in $\R^2$ with $d_{\R^2}(x_2',p') = d(x_2,p)$ and $d_{\R^2}(x_3',p') = d(x_3,p)$.

Recall that a metric space $(\Omega,d)$ is called 
\begin{itemize}
    \item metrically-complete if any Cauchy sequence of points in $\Omega$ has a limit in $\Omega$, and 
    \item geodesically-complete if any geodesic segment can be extended to a geodesic line $\R \rightarrow \Omega$.
\end{itemize}

For a locally-compact CAT(0) space $(\Omega,d)$, the followings are known (see \cite{Caprace2014Lectures}):
\begin{itemize}
    \item $(\Omega,d)$ is proper as a metric space if and only if it is metrically-complete, 
    \item $(\Omega,d)$ is geodesically-complete, then it is proper.
    \item $(\Omega,d)$ is homeomorphic to a finite dimensional manifold, then it is geodesically-complete (see \cite[Proposition 5.1.2 in Chapter II]{BridsonHaefligerMetricSpaces}).
\end{itemize}

We also introduce the boundary $\partial \Omega$ at infinity of $(\Omega,d)$ in the sense of \cite[Chapter II.8]{BridsonHaefligerMetricSpaces}.
An isometric map $c : [0,\infty) \rightarrow \Omega$ 
is called a geodesic ray issuing from a point $c(0) \in \Omega$.
Two geodesic rays $c, c' : [0,\infty) \rightarrow \Omega$ are said to be asymptotic, or equivalent, if 
there exists a constant $a \geq 0$ such that $d(c(t),c'(t)) \leq a$ for any $t \in [0,\infty)$. 
The set of all equivalent classes of geodesic rays is denoted by $\partial \Omega$.
The equivalent class of a geodesic ray $c$ is written as $c(\infty) \in \partial \Omega$.
Note that each isometry $\gamma$ on $\Omega$ causes a natural permutation on the set $\partial \Omega$. 
In particular, each isometric action on $\Omega$ by a group $G$ can be extended to a $G$-action on the set $\Omega^{-} := \Omega \sqcup \partial \Omega$.

\begin{remark}
The boundary $\partial \Omega$ and the union $\Omega^{-} := \Omega \sqcup \partial \Omega$ admit some important topologies and metrics.
For example, if $(\Omega,d)$ is a proper CAT(0) space 
and we consider the cone topology on $\Omega^{-} := \Omega \sqcup \partial \Omega$,
then $\Omega^{-}$ gives a compactification of $\Omega$. See \cite[Chapter II]{BridsonHaefligerMetricSpaces} for the details.
\end{remark}

\subsection{A sufficient condition for Property (S) of metric homogeneous $G$-spaces}

Let $(\Omega,d)$ be a proper (and thus, metrically-complete) CAT(0) space 
equipped with an isometric action of a group $G$,
and $M$ a closed $G$-orbit of $\Omega$.
Note that $M$ itself is a proper metric space with respect to the restriction of $d$ on $M$.

For each $\xi \in \partial \Omega$, 
we denote by $P_\xi$ the isotropy subgroup of $G$ at $\xi$,
that is,  
\[
P_{\xi} := \{ g \in G \mid g \cdot \xi = \xi \} \subset G.
\]

The following proposition gives a sufficient condition for Property (S) of such a metric $G$-space $M$:

\begin{theorem}\label{theorem:SinCATzero}
In the setting above, 
we assume that the following holds:
one can choose a constant $b \geq 0$ 
such that 
for any $p,q \in M$, 
there exists 
a geodesic ray $c$ in $\Omega$ 
issuing from $p$
such that 
\[
d(q,c) := \min_{t \in [0,\infty)}d(q,c(t)) \leq b
\]
and the $P_{c(\infty)}$-action on $M$ is transitive.
Then the $G$-space $M$ has Property (S). 
\end{theorem}

We give a proof of Theorem \ref{theorem:SinCATzero} by applying the following lemma:

\begin{lemma}[{\cite[Proposition 8.2 and its proof in Chapter II.8]{BridsonHaefligerMetricSpaces}}]\label{lemma:asymptoticCAT0}
Let us fix any $p,p' \in \Omega$
and any geodesic ray $c$ in $\Omega$ issuing from $p$.
Then there uniquely exists a geodesic ray $c'$ in $\Omega$ issuing from $p'$ asymptotic to $c$.
Furthermore, for any $t \in [0,\infty)$, 
\[
d(c(t),c'(t)) \leq d(p,p').
\]
\end{lemma}

\begin{proof}[Proof of Theorem \ref{theorem:SinCATzero}]
    Take $b \geq 0$ as in the assumption.
    Fix any $r \geq 0$.
    We put $\tau(r) := r + 2b$.
    Take any $p_0,p,q \in M$ with $d(p,p_0) \leq r$.
    Our goal is to prove that there exists $g \in G$ such that $g \cdot p = p_0$ and $d(q,g \cdot q) \leq \tau(r) = r + 2b$.
    By the assumption, one can find and do take a geodesic ray $c$ in $\Omega$ issuing from $p$ such that $d(q,c) \leq b$ and the $P_{c(\infty)}$-action on $M$ is transitive.
    Let us fix $t_0 \in [0,\infty)$ with $d(q,c(t_0)) \leq b$ and $g \in P_{c(\infty)}$ such that $g \cdot p = p_0$.
    Put $q_0 := g \cdot q \in M$.
    We only need to show that $d(q,q_0) \leq r + 2b$.
    Let us consider the geodesic ray $g \cdot c$ issuing from $p_0$ in $\Omega$ defined by 
    \[
    g \cdot c : [0,\infty) \rightarrow \Omega, ~ t \mapsto g \cdot c(t).
    \]
    Since $g \in P_{c(\infty)}$, we have that $g \cdot c$
    is asymptotic to $c$.
    Thus, by Lemma \ref{lemma:asymptoticCAT0},
    \[
    d(c(t_0),(g \cdot c)(t_0)) \leq d(p,p_0) \leq r.
    \]
    We note that 
    \[
    d((g \cdot c)(t_0),q_0) = d(c(t_0),q) \leq b.
    \]
    Therefore, we obtain 
    \begin{align*}
        d(q,q_0) 
            &\leq d(q,c(t_0)) + d(c(t_0),(g \cdot c)(t_0)) + d((g \cdot c)(t_0),q_0) \\
            &\leq r + 2b.
    \end{align*}
    This completes the proof.
\end{proof}

We also obtain the following corollary: 

\begin{corollary}\label{corollary:Ptransitive}
    Let $M$ be a geodesically-complete CAT(0) space equipped with an isometric transitive action of a group $G$.
    Assume that for each $\xi \in \partial M$, the isotropy subgroup $P_{\xi}$ of $G$ at $\xi$ acts on $M$ transitively.
    Then the metric $G$-space $M$ has Property (S).
\end{corollary}

\begin{remark}
It should be remarked that Caprace--Monod \cite[Theorem 1.3]{CapraceMonod2009Isometry} proved that 
for a geodesically-complete CAT(0) space $M$ equipped with an isometric action of a group $G$, 
if the isotropy group $P_\xi$ acts cocompactly on $M$ for any $\xi \in \partial M$, 
then $M$ is isometric to a product of symmetric spaces, Euclidean buildings and Bass--Serre trees.
\end{remark}

\section{Kobayashi's properness criterion from  Theorem \ref{theorem:main_section}}\label{section:KobayashiCriterion}

In this section, let us argue that 
Kobayashi's properness criterion (Theorem \ref{theorem:Kobayashi96})
can be explained by applying Theorem \ref{theorem:main_section}
to isometric actions of real linear reductive  Lie groups on non-compact Riemannian symmetric spaces.

Let $G$ be a real linear reductive Lie group.
We fix a realization of $G$ 
as a transpose-stable closed subgroup of $GL(N,\R)$ 
with finitely many connected components ($N \in \Z_{\geq 0}$).
The Cartan involution $\theta$ on $G$ is defined by $\theta(g) := {}^t (g^{-1})$ for $g \in G$,
where ${}^t A$ denotes the transpose of the matrix $A$.
Then the $\theta$-fixed point subgroup $K = G \cap O(n)$ of $G$ 
is a maximal compact subgroup of $G$.
The Lie algebra of $G$ is denoted by $\mathfrak{g}$ and we write $\mathfrak{g} = \mathfrak{k} + \mathfrak{p}$ for the Cartan decomposition of $\mathfrak{g}$ with respect to the Cartan involution $\theta$. 
Note that by the adjoint representation, which will be denoted by $\Ad : G \rightarrow GL(\mathfrak{g})$, the group $K$ acts on $\mathfrak{p}$.

We write $M := G/K$ for the homogeneous manifold of $(G,K)$ with the base point $* := eK \in M = G/K$, where $e$ is the unit of $G$.
The quotient map from $G$ to $M = G/K$ is denoted by $\pi : G \rightarrow M$.
As is well-known that the tangent space $T_*M$ of $M$ at the point $*$ can be identified with $\mathfrak{p}$ canonically.
Let us fix a $K$-invariant inner-product on $\mathfrak{p}$.
Then such the inner-product defines a 
$G$-invariant Riemannian metric on $M$.
Then it is well-known that $M$ equipped with such the metric is a simply-connected Riemannian globally symmetric space without compact factors (see \cite[Chapter V]{HelgasonDiff} for the details), 
In particular, 
$M$ has a direct product decomposition 
\[
M = M_0 \times M_{-}
\]
where $M_0$ is a Euclidean space and $M_-$ is a Riemannian globally symmetric space of non-compact type.
Note that for any parallel shift $t$ on the Euclidean space $M_0$ and any $h \in \Isom_0(M_{-})$, one can find $g \in G$ such that 
$g$ defines the isometry $(t,h)$ on $M = M_0 \times M_{-}$,
where $\Isom_0(M_{-})$ denotes the identity component of the full isometry group of $M_{-}$.

Since $M_0$ and $M_{-}$ are both geodesically-complete CAT(0) spaces (see \cite{Caprace2014Lectures}),
our $M = M_0 \times M_{-}$ is also a geodesically-complete CAT(0) space.
Let us denote by $\partial M = \partial M_0 \times \partial M_{-}$ the boundary at infinity of $M$ (see Section \ref{subsection:notationCATzero} for the definition of $\partial M$).
Then for each $\xi = (\xi_0,\xi_{-}) \in \partial M = \partial M_0 \times \partial M_{-}$, 
the fixed point subgroup $P_\xi := \{ g \in G \mid g \cdot \xi = \xi \}$ acts on $M$ transitively.
In fact, 
on the Euclidean space $M_0$, 
the group of parallel shifts acts transitively on $M_0$ and fixes the boundary $\partial M_0$  pointwisely,
and on $M_{-}$, the fixed point subgroup at $\xi_{-}$ of the connected linear semisimple Lie group $\Isom_0(M_-)$ is a parabolic subgroup and acts on $M_-$ transitively (See \cite[\S I.1 and I.2]{BorelJiCompSLS} for the details).
Therefore, by Corollary \ref{corollary:Ptransitive}, 
our $G$-space $M$ has Property (S).

In addition, let us choose a maximal abelian subspace $\mathfrak{a}$ of $\mathfrak{p}$,
and consider the inner-product on $\mathfrak{a}$ induced by the fixed inner-product on $\mathfrak{p}$.

We define the finite group $W$ by 
\begin{align*}
N_K(\mathfrak{a}) &:= \{ k \in K \mid \Ad(k) \mathfrak{a} = \mathfrak{a} \} \subset K, \\
Z_K(\mathfrak{a}) &:= \{ k \in K \mid \Ad(k) X = X \text{ for each } X \in \mathfrak{a} \} \subset N_K(\mathfrak{a}),  \\
W &:= N_K(\mathfrak{a})/{Z_K(\mathfrak{a})}.
\end{align*}
Then $W$ can be considered as a finite subgroup of the orthogonal group $O(\mathfrak{a})$ on the inner-product space $\mathfrak{a}$.
Note that if $G$ is contained in a connected complex Lie group $G_\C$, 
then the finite group $W$ defined above coincides with the Weyl group of the restricted root system for $(\mathfrak{g},\mathfrak{a})$.

We define $A := \exp(\mathfrak{a}) \subset G$ and consider the connected, closed, totally-geodesic flat submanifold $\pi(A) = A \cdot *$ of $M$.
Note that the bijective map 
\[
\iota : \mathfrak{a} \rightarrow \pi(A), ~X \mapsto \exp(X) \cdot *
\]
is isometric, and thus the finite group $W$ also acts on $\pi(A)$ isometrically.
It is well-known (see \cite[Example 2.1 (iv)]{Thorbergsson2005Transformation}) that $\pi(A)$ is a section in the sense of the last part of Section \ref{subsection:Korbits}, 
and the generalized Weyl group coincides with the finite group $W$ defined above.

Therefore, Theorem \ref{theorem:Kobayashi96} can be obtained by applying Theorem \ref{theorem:main_section} to 
the section $\Sigma = \pi(A) \cong \mathfrak{a}$.
Note that for pairs of non-empty subsets $(C_1,C_2)$ of $\mathfrak{a}$, 
$C_1 \sim C_2$ as in a locally-compact additive group $\mathfrak{a}$ 
if and only if 
$C_1 \sim C_2$ as in a proper metric space $\mathfrak{a}$.
Therefore, $(C_1,C_2)$ is proper in the group $\mathfrak{a}$ if and only if $(C_1,C_2)$ is HBI in the proper metric space $\mathfrak{a}$.

\section*{Acknowledgements.}
The second author obtained his PhD degree under the supervision of Professor Toshiyuki Kobayashi, and learned many things about proper actions on homogeneous spaces.
Without his guidance, our studies would not have been possible.
We are also indebted to Hiroshi Tamaru, Akira Kubo, Koichi Tojo, Kazuki Kannaka and Kaede Aoyama for many helpful comments.

The first author is supported by JST SPRING, Grant Number JP-MJSP2132.
The second author is supported by JSPS Grants-in-Aid for Scientific Research JP20K03589, JP20K14310, and JP22H0112.


\providecommand{\bysame}{\leavevmode\hbox to3em{\hrulefill}\thinspace}
\providecommand{\MR}{\relax\ifhmode\unskip\space\fi MR }
\providecommand{\MRhref}[2]{%
  \href{http://www.ams.org/mathscinet-getitem?mr=#1}{#2}
}
\providecommand{\href}[2]{#2}

\end{document}